\documentclass[12pt]{asl}

\usepackage{ABT}

\newcommand{\ee}{\supseteq_{\text{ee}}}
\newcommand{\eeby}{\subseteq_{\text{ee}}}
\newcommand{\ees}{\supset_{\text{ee}}}
\newcommand{\eebys}{\subset_{\text{ee}}}

\newtheorem{claim}[thm]{Claim}

\title{A variant proof of $\Con(\frb<\fra)$}

\author{J\"org Brendle and Andrew D. Brooke-Taylor
}

\date{\today}

\address{Group of Logic, Statistics and Informatics,\\
Graduate School of System Informatics,\\
Kobe University\\
Rokko-dai 1-1,\\
Nada, Kobe, 657-8501\\
Japan}

\email{brendle@kurt.scitec.kobe-u.ac.jp}

\thanks{Partially supported by Grant-in-Aid for Scientific Research (C) 24540126, 
   Japan Society for the Promotion of Science.}

\email{andrewbt@kurt.scitec.kobe-u.ac.jp}

\thanks{Written while holding a
JSPS Postdoctoral Fellowship for Foreign Researchers at Kobe University and
supported by JSPS Grant-in-Aid no. 23 01765.}

\begin{document}

\begin{abstract}
We present a variation of the proof in \cite{Bre:MFMF} of 
$\Con(\frb<\fra)$, which in particular removes some of the obstacles
to generalising the argument to cardinals $\kappa>\omega$.
\end{abstract}

\maketitle

\section{Introduction}
The generalisations of cardinal characteristics of the continuum to cardinals
$\ka$ greater than $\omega$ has generated significant interest recently.
A particular result that has so far resisted attempts at generalisation is
the statement that $\frb<\fra$ is consistent.
Blass, Hyttinen and Zhang \cite[Section~5]{BHZ:MFN}
briefly survey the different approaches known for proving $\Con(\frb<\fra)$,
highlighting the difficulties each presents for a generalisation.

We present here a variation on the proof of $\Con(\frb<\fra)$
given in \cite{Bre:MFMF}, which we hope will be more amenable to 
generalisation.  In particular, the proof in \cite{Bre:MFMF} relies on a
rank argument, which of course cannot be na\"{\i}vely generalised to 
uncountable $\ka$.  
We show here that it may be replaced by a suitable formulation
in terms of games, which \emph{does} generalise to higher $\ka$.
Indeed, with this observation, the question of forcing $\frb_\ka>\fra_\ka$
for some suitable large cardinal $\ka$ seems to boil down to interesting
questions about the existence of suitable filters on $\ka$.



\section{Preliminaries}

Let $\kappa$ be an infinite cardinal. A family $\calA \subseteq [\kappa]^\kappa$
is called \emph{almost disjoint} if $| A \cap B| < \kappa$ for any two distinct
members $A$ and $B$ of $\calA$. $\calA$ is a \emph{maximal almost disjoint
family} (\emph{mad family}, for short) if $\calA$ is almost disjoint and maximal with
this property. This means that for every $C \in [\kappa]^\kappa$ there is
$A \in \calA$ such that $|A \cap C | =\kappa$. The \emph{almost disjointness
number} $\fra_\kappa$ is the least size of a mad family on $\kappa$ of
size at least $cf (\kappa)$ (equivalently, of size $> cf (\kappa)$). 
In case $\kappa = \omega$ write $\fra$ for $\fra_\omega$.

Now assume $\kappa$ is a regular cardinal. For functions $f,g \in \kappa^\kappa$,
say that $g$ \emph{eventually dominates} $f$ ($f \leq^* g$ in symbols) if 
$f(\alpha ) \leq g(\alpha)$ holds for all $\alpha$ beyond some $\alpha_0 < \kappa$.
The \emph{unbounding number} $\frb_\kappa$ is the least size of an
unbounded family $\calF$ in the order $(\kappa^\kappa, \leq^*)$. That is,
for all $g \in \kappa^\kappa$ there is $f \in \calF$ with $f (\alpha) > g(\alpha)$ for
cofinally many $\alpha$'s. Again we write $\frb$ instead of $\frb_\omega$.

Let $\calF$ be a filter on $\omega$. \emph{Mathias forcing $\bbM (\calF)$ with
$\calF$} consists of conditions $(s,F)$ such that $s \in [\omega]^{<\omega}$,
$F \in \calF$, and $\max (s) < \min (F)$. $\bbM (\calF)$ is ordered by
$(t,G) \leq (s,F)$ if $s \subseteq t \subseteq s \cup F$ and $G \subseteq F$.
It is well-known and easy to see that $\bbM (\calF)$ is a $\sigma$-centered
forcing which introduces a pseudointersection $Z$ of the filter $\calF$.
This means that $Z \subseteq^* F$ for all $F \in \calF$, where $\subseteq^*$
denotes \emph{almost inclusion}: $A \subseteq^* B$ iff $A \smallsetminus B$ is finite.

In \cite{Bre:MFMF}, the notion of \emph{pseudocontinuity} is used.
This notion and the corresponding basic lemma can 
be nicely phrased in terms of continuity with respect to an
appropriate topology.

\begin{defn}
The \emph{initial segment topology} on $\omega$ is the topology which has
the (von Neumann) ordinals as open sets.
We denote $\omega$ endowed with this topology by $\omega_i$.
\end{defn}

\begin{defn}
A function to $\omega$ or $\omega^\omega$ is \emph{pseudocontinuous}
if it is continuous as a function to $\omega_i$ or $\omega_i^\omega$ 
respectively.
\end{defn}

Thus, a pseudocontinuous function $F:X\to\omega$ is one such that
for every $n\in\omega$, the set of $x$ in $X$ with image
at most $n$ is open.

\begin{lem}\label{pctsimbdd}
Compact sets in $\omega_i$ and $\omega_i^\omega$ are bounded.
In particular, any pseudocontinuous image in $\omega$ or $\omega^\omega$ 
of a compact set must be bounded.
\end{lem}
\begin{proof}
The Lemma is clear for $\omega_i$.
Similarly,
compact $K\subset\omega_i^\omega$ are in fact bounded in the strict
(not just $\leq^*$) sense.
Otherwise, there would be some $m$ in $\omega$
such that $f(m)$ is unbounded in $\omega$ for $f\in K$, and then
the open sets 
\(
\mathcal{O}_{m,n}=\{f\in\omega_i^\omega\st f(m)\leq n\}
\)
for $n<\omega$
would form an open cover of $K$ with no finite subcover.
\end{proof}

As usual we may identify $\Power(\omega)$ with $2^\omega$ by way of the map 
taking sets to their characteristic functions, $\chi:X\mapsto\chi_X$.
We give $\Power(\omega)$ the corresponding topology, making $\chi$ a
homeomorphism from $\Power(\omega)$ to the Cantor space $2^\omega$.

\begin{defn}
For any cardinal $\la$, we call a filter $\calG\subseteq\Power(\omega)$
a \emph{$K_\la$-filter} if
it is generated by the union of fewer than $\la$ many 
compact subsets of $\Power(\omega)$.  We write $K_\sigma$ for $K_{\aleph_1}$.
\end{defn}

\begin{lem}\label{compactops}
If $K_{0},\ldots, K_{n-1}$ are (finitely many)
compact subsets of $\Power(\omega)$, then the pointwise intersection 
\[
\bigwedge_{i<n}K_i=\Big\{\bigcap_{i<n}G_i\ \Big|\ 
(G_0,\ldots,G_{n-1})\in\prod_{i<n}K_i\Big\}
\]
and the pointwise union
\[
\bigvee_{i<n}K_i=\Big\{\bigcup_{i<n}G_i\ \Big|\ 
(G_0,\ldots,G_{n-1})\in\prod_{i<n}K_i\Big\}
\]
are compact.
Furthermore, for any compact set $K\subseteq\Power(\omega)$, the
upward closure
\[
\bar K=\{A\in\Power(\ka)\st\exists B\in K(A\supseteq B)\}
\]
is also compact.
\end{lem}
\begin{proof}
The product $\prod_{i<n}K_i$ is compact by the Tychonoff theorem, and
the functions $\Power(\omega)^n\to\Power(\omega)$ given by
$(G_0,\ldots,G_{n-1})\mapsto\bigcap_{i<n}G_i$ and 
$(G_0,\ldots,G_{n-1})\mapsto\bigcup_{i<n}G_i$ are clearly continuous,
so $\bigwedge_{i<n}K_i$ and $\bigvee_{i<n}K_i$ are compact.
Finally, for compact $K\subseteq\Power(\omega)$, $\bar K$ is just 
$K\vee\Power(\omega)$.
\end{proof}


\section{The proof}

We work in 
a model $V$ of ZFC in which $\la=\frc^V$ is
a regular cardinal satisfying $2^\la=\la^+$, 
and there is an unbounded, $<^*$-well-ordered sequence
$\langle f_\al:\al<\la\rangle$ 
of strictly increasing functions 
from $\omega$ to $\omega$.
For example, any model of GCH will suffice as a ground model, 
and these properties will be preserved in intermediate stages of our forcing
iteration.

Let $\calA$ be an infinite maximal almost disjoint family in $V$ of subsets of $\omega$.
\begin{thm}\label{PA}
There is a ccc forcing $\P(\calA)$ such that 
\[
\forces_{\P(\calA)}\calA\text{ is not mad and }
\langle f_\al:\al<\la\rangle\text{ is still unbounded.}
\]
\end{thm}
\begin{proof}
Let $\calF=\calF(\calA)$ be the dual 
filter of $\calA$, 
that is, the 
filter generated by the sets whose complements are finite or
in $\calA$.
Note that this filter is proper: if for some $k<\omega$ there were 
$\{A_i\st i<k\}\subset\calA$ such that
$|\bigcap_{i<k}\omega\smallsetminus A_i|<\omega$, 
any other element of $\calA$ would have 
infinite intersection with one of the $A_i$,
violating almost disjointness.
Note that the generic subset of $\omega$ introduced by Mathias forcing with
$\calF$, or any filter extending $\calF$, will end the madness of $\calA$,
as it will be almost contained in $\omega\smallsetminus A$ for every 
$A\in\calA$.

First we add $\lambda$ many Cohen reals. It is well-known that
the unboundedness of $\langle f_\al:\al<\la\rangle$ is preserved 
in this intermediate extension. In case $\calA$ is not mad anymore
in this extension we are done. Also, if $\calF$ is contained in
a $K_\lambda$ filter $\calG$ in the intermediate extension, we may simply
force with $\bbM (\calG)$ for it is well-known, and easy to see~\cite[3.2]{Bre:MFMF},
that Mathias forcing with a $K_\lambda$-filter does not destroy
the unboundedness of $\langle f_\al:\al<\la\rangle$. So assume
that $\calF$ is not contained in any $K_\lambda$-filter.

We shall recursively construct a 
filter $\calG\supseteq\calF$ such that furthermore
\begin{equation}
\forces_{\bbM(\calG)}\langle f_\al:\al<\la\rangle\text{ is unbounded}.\tag{$*$}
\end{equation}
Along the construction we shall take care of every potential
$\bbM(\calG)$-name for a function in $\omega^\omega$, 
either ``killing it'' or
``sealing it off''.

To be precise: let us refer to partial functions 
$\tau:[\omega]^{<\omega}\times\omega\dashrightarrow\omega$ as \emph{preterms},
and let $\calT=\{\tau_\be:\be<\la\}$ 
be an enumeration of the set of all preterms.
Note in particular that if $\calG\supseteq\calF$ is a filter and 
$\dot g$ is an $\bbM(\calG)$-name for a function in $\omega^\omega$, then
$\tau=\tau_{\dot g}$ given by
\[
\tau(s,m)=n\text{ iff }
\exists G\in\calG\left((s,G)\forces\dot g(m)=n\right)
\]
is a preterm, the \emph{preterm associated with $\dot g$}.
We shall constrain attention to names $\dot g$ such that 
$\bbone\forces_{\bbM(\calG)}\dot g\in\omega^\omega$, 
since every function from $\omega$ to $\omega$ in
the generic extension has such a name; we call such names \emph{total names}.

We construct filters $\calG_\be$ for
$0\leq\be\leq\la$, starting from $\calG_0=\calF$, such that
\renewcommand{\theenumi}{\roman{enumi}}
\begin{itemize}
\item for each $\be<\la$, $\calG_{\be+1}$ is generated by $\calG_\be$ and a
$K_{\sigma}$ filter $\calH_\be$, 
\item $\calG_\de=\bigcup_{\be<\de}\calG_\be$ for each limit ordinal $\de\leq\la$,
\end{itemize}
\renewcommand{\theenumi}{\arabic{enumi}}
and either 
\begin{description}
\item[(KILL)] for all filters $\calH\supseteq\calG_{\be+1}$,
$\tau_\be$ is not associated with any total $\bbM(\calH)$-name, or 
\item[(SEAL)] 
there is an $\al<\la$ such that
for all filters $\calH\supseteq\calG_{\be+1}$ and all
$\bbM(\calH)$-names $\dot g$, if
$\tau_{\dot g}=\tau_\be$ then
$\forces_{\bbM(\calH)}
\dot g\ngeq^*\check f_\al$.
\end{description}
Clearly any 
filter $\calG\supseteq\calG_\la$ will then satisfy ($*$).

So suppose $\calG_\be$ has been defined for some $\be<\la$; we wish to find
an appropriate $K_\sigma$ filter $\calH_\be$.
Note that $\calG_\be$ is generated by $\calF$ and a  
$K_\la$ filter $\calG'_\be$; 
without loss of generality we may assume that $\calF$ contains all 
cofinite subsets of $\omega$.  
Let $\calK_\be$ be a family of fewer than $\la$ many 
compact subsets of $2^\omega$
generating $\calG'_\be$.
By Lemma~\ref{compactops},
we may assume that $\calK_\be$ is
closed under finite pointwise intersections,
and that for all $K\in \calK_\be$, 
$K$ is upwards-closed under $\subseteq$, so that
$\calG_\be'=\bigcup\calK_\be$.

Everything that has come so far can actually be considered to have
occurred in a partial extension model, 
between the original model and the full extension with $\la$-many Cohens.
More explicitly, all (codes of) elements of $\calK_\be$ belong to this
intermediate model.

Let $\eebys$ denote the strict 
end-extension relation on $[\omega]^{<\omega}$: that is,
$s\eebys s'$ if and only if $s\subset s'$ and 
$\max(s)<\min(s'\smallsetminus s)$; define $\eeby$, $\ees$ and $\ee$ 
accordingly.

In \cite{Bre:MFMF}, a rank function was used.
For our generalisation, we take a different approach using games,
but use these games to much the same end as the rank function is used
in \cite{Bre:MFMF}.
It should be noted that our games are very closely related to the games
independently introduced by Guzm\'an, Hru\v s\'ak, and Mart\'inez~\cite{GHM:CF},
also in the context of a proof of $\Con(\frb<\fra)$.

Let $\tau = \tau_\beta$.

\begin{defn}\label{nominalisation}
Given $\tau\in\calT$, the 
\emph{$\tau$ nominalisation exercise} is the following game.
There are two players, Sensei and Student.
On turn 0, Sensei chooses an $m\in\omega$ and
$t_0\in[\omega]^{<\omega}$.
At odd stages $2d+1$, 
Student plays a filter set
$F( d)\in\calF$
and a 
compact set $K( d)\in\calK_\beta$.
At even stages $2d+2$, 
Sensei plays an element $t_{d+1}$ of $[\omega]^{<\omega}$
such that
\begin{itemize}
\item $t_{d+1}$ end-extends $t_d$
\item $t_{d+1}\smallsetminus t_d\subseteq  F(d)$
\item $t_{d+1}\smallsetminus t_d$ meets every member of $K(d)$.
\end{itemize}
If there is $s\subseteq t_{d+1}$
end extending $t_0$ such that $(s,m)\in\dom(\tau)$,
Sensei declares Student to have passed and the game ends.
If the game continues for infinitely many stages, then (clearly) Student
has failed.
\end{defn}

Note that, since $\calG_\beta$ is a filter, and by compactness of $K(d)$,
a $t_{d+1}$ satisfying the requirements always exists. Also notice that if Student wins, he
wins after finitely many steps. Hence the game is open and, by the classical
Gale-Stewart Theorem, determined.


As in \cite{Bre:MFMF}, we now distinguish two cases 
(in \cite{Bre:MFMF} they are \emph{Subcases}), corresponding to 
options (KILL) and (SEAL) above. 

\subsection{Case a.}  
There are $m\in\omega$ and $t_0\in [\omega]^{<\omega}$
such that Sensei has a winning strategy in the $\tau$ nominalisation exercise
with 0th move $(m,t_0)$:
play will continue for infinitely many steps.
In this case we shall choose $\calH_\be$ in such a way that (KILL) holds:
$\tau$ will not correspond to a name for a function $\omega\to\omega$ in
the generic extension.  The reader may wish to remember which case is which
by the mnemonic ``the $\tau$ that can be named is not the eternal $\tau$.''

We shall actually work in the extension of such the intermediate
model by one further Cohen function $c:\omega\to\omega$.

Consider the tree $T$ of all possible sequences of plays 
$(t_0,t_1,t_2,\ldots)$ for Sensei according to his strategy, corresponding to
all possible plays of Student.
Note that $T$ is infinitely branching since $\calF$ extends the Frechet
filter.
Use the Cohen function $c$ to choose a branch through $T$, and denote the
union of the $t_i$ of this branch by $G$.
There is no $(s,m)$ with $m$ from Sensei's first move and
$t_0\eeby s\subseteq G$ such that $(s,m)\in\dom(\tau_\be)$.
Indeed otherwise, the $\tau_\be$ nominalisation exercise would have ended once Sensei
played $t_d$ sufficiently long to cover $s$.
Thus, for any filter $\calH\ni G$, $\tau\neq\tau_{\dot g}$ for any total
$\bbM(\calH)$ name $\dot g$.
We may therefore simply take $\calH_\be=\{G\}$ in order to satisfy (KILL).
To check that $\{G\}\cup\calG_\be$ generates a filter,
consider any $F\in\calF$ and $G'\in\calG'_\be$, say $G'$ is in the 
compact set $K\in\calK_\be$.
For every $t_d\in T$, there is a successor node $t_{d+1}$ in the tree $T$ 
that is Sensei's response, according to his strategy, to Student playing $F$ and $K$, 
and so in particular this
$t_{d+1}$ meets the intersection of $F$ and every member of $K$.
Thus, by Cohen genericity we have that $|G\cap F\cap G'|=\omega$,
completing Case a. (Note that $G'$ may not belong to the intermediate model;
this, however, is irrelevant for it is sufficient that $K$ does. By genericity the
Cohen real $c$ will produce infinitely many $d$ such that $t_{d+1} \smallsetminus t_d$
is contained in $F$ and meets every $G'' \in K$, and this is clearly absolute and
thus also holds for $G'$.)

\subsection{Case b.}  The negation of Case a: 
for every 0th move $(m,t_0)$ by Sensei, Student has a winning strategy in the $\tau_\be$
nominalisation exercise.
In this case we wish to choose $\calH_\be$ in such a way that (SEAL)
holds.


Since Sensei chooses his moves from a countable set, there are clearly
only countable many filter sets $F_\ell \in \calF$, $\ell\in\omega$, which appear as $F(d)$ in
some $2d +1$st move of Student playing according to his strategy.

Suppose that for all but less than $\lambda$ many members $A$ of $\calA$,
there is $G \in \calG'_\beta$ such that $A \cap G$ is finite. Then, adding less
than $\lambda$ many sets of the form $\omega \smallsetminus A$, $A \in \calA$, to $\calG'_\beta$
results in a $K_\lambda$ filter containing $\calF$. This contradicts our initial assumption.
Hence, for $\lambda$ many $A \in \calA$, $A \cap G$ is infinite for all $G \in \calG'_\beta$.
Let $A_j$, $j \in \omega$, be countably many such $A$'s such that for
each $j$ and $\ell$, $A_j$ is almost contained in $F_\ell$: this is possible
because $\calF$ is the dual filter of the mad family $\calA$. 
 
For each $G'\in\calG'_\be$, $k\in\omega$, $j\in\omega$, and finite subset $T$ of
$[\omega]^{<\omega}$,
we define a function
$f_{G',k,j,T}:\omega\to\omega$ as follows.
\begin{align*}
f_{G',k,j,T}(m)=
\min\{n\st&
\text{ for any partition }A_j=\bigcup_{i<k}B_i
\text{ there is }i<k
\text{ s.t.}\\
&\forall t\in T \exists s\ees t(
s\smallsetminus t\subseteq B_i\cap G'\land
\tau_\beta(s,m)\leq n)\}.
\end{align*}

\begin{lem}\label{fwdefnd}
For every $G'\in\calG'_\be$, $k,j\in\omega$, and $T\in[[\omega]^{<\omega}]^{<\omega}$,
$f_{G',k,j,T}$ is well-defined.  
\end{lem}
\begin{proof}
Fix $m\in\omega$.  Given a partition $\{B_i\st i<k\}$ of ${A_j}$,
let ``$n$ suffices for $\{B_i\st i<k\}$'' mean the natural thing
in the context of the definition of $f_{G',k,j,T}$,
namely, 
that there is $i<k$
such that for every $t\in T$
there is $s\ees t$
with
$s\smallsetminus t\subseteq B_i\cap G'$ and
$\tau_\beta(s,m)\leq n$.
So now fix a partition $\{B_i\st i<k\}$ of ${A_j}$;
we shall show that there is a $ n\in\omega$ that suffices for it.
Let $i<k$ be such that $|B_i\cap G'\cap G|=\omega$ for every 
$G\in\calG_\be'$:
such an $i$ must exist,
since ${A_j}$ has infinite intersection with every member of 
the filter $\calG_\be'$.
Finally, fix $t\in T$.

Consider a play of the $\tau_\be$ naming exercise in which  Student follows his strategy,
Sensei's 0th move is $(m, t_0)$ with $t_0 = t$, and his later moves 
always satisfy the additional requirement $t_{d+1} \smallsetminus t_d \subseteq B_i \cap G'$. 
Since $B_i$ is almost contained in all $F(d)$ played by Student according to his
strategy and since $B_i$ has infinite intersection with all $G \in \calG'_\beta$,
Sensei always has a valid such move. 


So we have that eventually Sensei plays a $t_d$ such that 
\[
\exists n_t\in\omega\exists s\subseteq t_d(s\ees t\land \tau_\be(s,m)=n_t).
\]
Of course, by the construction of the game,
$s\smallsetminus t_0\subseteq B_i\cap G'$.
Taking such an $n_t$ for each $t\in T$ and setting
$n=\max_{t\in T}(n_t)$, we have that $n$ suffices for $\{B_i\st i<k\}$.

Now, with $k$ still fixed but allowing the partition 
$\{B_ i\st i<k\}$ to vary, let us denote by $ n(\{B_i\st i<k\})$ 
the least $n$ that suffices for $\{B_ i\st i<k\}$.
The space of partitions of ${A_j}$ into $k$ pieces can be identified
with $k^{A_j}$ and thus 
when endowed with the product topology is a compact
topological space.  Moreover, with this topology on the space of partitions,
the function $ n$
sending $\{B_ i\st i<k\}$ to $ n(\{B_ i\st i<k\})$ 
is clearly pseudocontinuous,
since $ n$ being sufficient for $\{B_ i\st i<k\}$ is witnessed by
finitely many finite tuples $s\smallsetminus t$ 
from $B_ i$, which of course define an open set in $k^{A_j}$.
Thus by Lemma~\ref{pctsimbdd} the image of the function $n$ is bounded below 
$\omega$.
The least such upper bound will be $f_{G',k,j,T}(m)$, and it follows
that $f_{G',k,j,T}$ is well-defined.
\end{proof}

\begin{lem}\label{unbddal}
There exists an $\al<\la$ such that for all $G'\in\calG'_\be$, $k,j\in\omega$ and
$T\in[[\omega]^{<\omega}]^{<\omega}$, $f_\al\nleq^*f_{G',k,j,T}$.
\end{lem}
\begin{proof}
We first note that, given $k,j$, $T$, and compact $K \in \calK_\beta$,
the function $f_{\cdot,k,j,T}$
sending $G'$ to $f_{G',k,j,T}$ is pseudocontinuous
from $K$ to $\omega^\omega$,
by much the same argument as in the proof of Lemma~\ref{fwdefnd}.
Indeed, fixing $m$ and $n$,
$\{ G' \st f_{G',k,j,T} (m) \leq n \}$ is open in $K$.


We thus have from Lemma~\ref{pctsimbdd} that for each $K\in\calK_\beta$,
$f_{\cdot,k,j,T}``K$ is bounded in $\omega^\omega$, say by $h_K$.
Since $\calK_\beta$ has fewer than $\la$ many elements, there is an $\al<\la$
such that $f_\al$ is not eventually dominated by any of the $h_K$,
and hence not by any $f_{G',k,j,T}$.
\end{proof}

We now show that $\al$ as given by Lemma~\ref{unbddal} will make
(SEAL) hold for an appropriate choice of $\calH_\be$.
Given
$t\in[\omega]^{<\omega}$, $G\in\Power(\omega)$, and $m\in\omega$, let
\[
g^\be_{t,G}(m)=
\min\{n\st \exists s\ee t(s\smallsetminus t\subseteq G\land
\tau_\be(s,m)=n)\}
\]
if the set on the right hand side is non-empty, and otherwise
put $g^\be_{t,G}(m)=\omega$.  
Thus, $g^\be_{t,G}$ is a function in $(\omega+1)^\omega$.
Let $\al<\la$ be such that $f_\al$ is not dominated by any 
$f_{G',k,j,T}$, as given by Lemma~\ref{unbddal}, and define
\[
\calH_\be=\{H\subseteq\omega\st \exists t\in[\omega]^{<\omega}
(g^\be_{t,\omega\smallsetminus H}\geq^*f_{\al})\}.
\]
Note that given $t\in[\omega]^{<\omega}$ and $m_0\in\omega$, the set
\[
\big\{H\subseteq\omega\st\forall m\geq m_0
\big(g^\be_{t,\omega\smallsetminus H}( m)\geq f_{\al}( m)\big)\big\}
\]
is closed in $\Power(\omega)$, 
and hence compact. 
Therefore, $\calH_\be$ is a $K_{\sigma}$ set.


To see that this set is an appropriate choice of $\calH_\be$ as called for 
above, we check the following.

\begin{claim}
Any filter $\calH\supseteq\calH_\be$ satisfies (SEAL).
\end{claim}
\begin{proof}
Let $\calH\supseteq\calH_\be$ be a filter, and assume 
$\tau_\be=\tau_{\dot g}$ for some $\bbM(\calH)$-name $\dot g$ for a function
in $\omega^\omega$.
Suppose there were $(t,G)\in\bbM(\calH)$ and $m_0\in\omega$ such that
\[
(t,G)\forces_{\bbM(\calH)}
\forall m\geq\check m_0\,\big(\dot g( m)\geq\check{f}_\al( m)\big).
\]
By the definition of $g^\be_{t,G}$, we must then also
have $g^\be_{t,G}( m)\geq f_\al( m)$ for all $ m\geq m_0$.
So $\omega\smallsetminus G\in \calH_\be\subseteq\calH$, contradicting the fact
that $\calH$ is a filter.
\end{proof}

\begin{claim}\label{HbeGbekacplt}
$\calH_\be\cup\calG_\be$ generates a filter.
\end{claim}
\begin{proof}
We take $F \in \calF$, $G'\in\calG_\be'$, and for some $k<\omega$, 
$H_i\in\calH_\be$ for $i<k$, and argue that
$F\cap G'\cap\bigcap_{i<k}H_i$ has cardinality $\omega$.
Assume for the sake of contradiction that 
$F\cap G'\subseteq^*\bigcup_{i<k}\omega\smallsetminus H_i$.
For each $i<k$, fix $t_i\in[\omega]^{<\omega}$ such that
$g_{t_i,\omega\smallsetminus H_i}^\be\geq^*f_\al$.
Also fix $j$ such that $A_j \subseteq^* F$.
Without loss of generality, we may take $a<\omega$ such that $A_j \smallsetminus a \subseteq F$,
$F\cap G'\smallsetminus a\subseteq\bigcup_{i<k}\omega\smallsetminus H_i$
and 
$\max(t_i)\geq a$ for every $i<k$
(if necessary by extending each $t_i$ with a sufficiently large element
of $\omega\smallsetminus H_i$: this can only increase the values of 
$g_{t_i,\omega\smallsetminus H_i}^\be$).
Fix $m_0\in\omega$ such that 
$g_{t_i,\omega\smallsetminus H_i}^\be(m)\geq f_\al(m)$
for all $m\geq m_0$ and $i<k$.
Let $T=\{t_i\st i<k\}$ and let 
$\{B_i\st i<k\}$ be
a partition of $A_j$ such that 
$B_i \cap G' \smallsetminus a\subseteq\omega\smallsetminus H_i$ for all $i<k$.
By the definition of $f_\al$, there is some $m> m_0$ such that
$f_\al(m)>f_{G',k,j,T}(m)$; take such a $m$, and 
denote $f_{G',k,j,T}(m)$ by $n$.
By the definition of $f_{G',k,j,T}$, there is an $i$ 
such that for all $t\in T$, there is $s\ees t$ such that 
$\tau_\be(s,m)\leq n$ and
$s\smallsetminus t$ is a subset of the intersection of $G'$ and $B_i$. In particular,
$\min(s\smallsetminus t_i)>\max(t_i)\geq a$,
$s\smallsetminus t_i\subset B_i \cap G'$,
and $\tau_\be(s,m)\leq n$.
Thus $s\smallsetminus t_i\subseteq\omega\smallsetminus H_i$,
from which we have $g_{t_i,\omega\smallsetminus H_i}^\beta(m)\leq n<f_\al(m)$,
contradicting the choice of $m_0$.
\end{proof}

This completes the construction of $\calG_{\be+1}$ from $\calG_\be$,
and hence the proof of Theorem~\ref{PA}.
\end{proof}

We are now ready for the consistency of $\frb <\fra$.
Recall from 
the beginning of this section that our 
ground model $V$ satisfies $\frc = \lambda$ is regular, $2^\lambda = \lambda^+$,
and $\langle f_\al:\al<\la\rangle$ is unbounded $<^*$-well-ordered.

\begin{thm}
There is a ccc forcing $\P$ such that 
\[
\forces_{\P} \fra = \lambda^+ \text{ and }
\langle f_\al:\al<\la\rangle\text{ is still unbounded.}
\]
In particular, $\frb \leq \lambda < \lambda^+ = \fra$ is consistent.
\end{thm}
\begin{proof} 

Perform a finite support iteration of orderings of type $\P (\calA)$ of
length $\lambda^+$, going through all (names for) mad families along
the way by a bookkeeping argument (this is possible by
the assumption $2^\lambda 
= \lambda^+$). The unboundedness of $\langle f_\al:\al<\la\rangle$ 
is preserved in the successor step of the iteration by Theorem~\ref{PA}
and in the limit step, by standard preservation results.
\end{proof}

\newpage

\bibliographystyle{asl}
\bibliography{ABT}

\end{document}